\documentclass[11pt, reqno]{amsart}
\usepackage{amsmath, amsfonts, amsthm, amssymb, multicol, mathtools, dsfont, mathrsfs}

\usepackage{graphicx}
\usepackage{float, hyperref}
\usepackage{scalerel,stackengine, subcaption}
\usepackage[usenames,dvipsnames,x11names]{xcolor}
\usepackage{enumitem}
\usepackage{pgfplots}
\usepgfplotslibrary{fillbetween}
\pgfplotsset{width=10cm,compat=1.9}
\usepackage[square,sort,comma,numbers]{natbib}
\allowdisplaybreaks

\makeatletter
\g@addto@macro\bfseries{\boldmath}
\makeatother

\makeatletter
\def\@setauthors{%
  \begingroup
  \def\thanks{\protect\thanks@warning}%
  \trivlist
  \centering\footnotesize \@topsep30\p@\relax
  \advance\@topsep by -\baselineskip
  \item\relax
  \author@andify\authors
  \def\\{\protect\linebreak}

  \normalsize\lowercase{\authors}%
  
	\ifx\@empty\contribs
  \else
    ,\penalty-3 \space \@setcontribs
    \@closetoccontribs
  \fi
  \endtrivlist
  \endgroup
}
\def\@settitle{\begin{center}
\LARGE\lowercase{\@title}
  \end{center}%
}
\makeatother
\newcommand{\authoremail}[1]{\email{\href{mailto:#1}{\color{lightblue}{#1}}}}
\newcommand{\authoraddress}[1]{\address{\normalfont{#1}}}

\numberwithin{equation}{section}

\newtheorem{thm}{Theorem}[section]
\newtheorem{lma}[thm]{Lemma}
\newtheorem{cor}[thm]{Corollary}

\newtheorem{prop}[thm]{Proposition}

\renewcommand{\epsilon}{\varepsilon}

\newcommand{\rd}{\mathbb{R}^d}

\renewcommand{\geq}{\geqslant}
\renewcommand{\leq}{\leqslant}

\newcommand{\fd}{\mathbb{F}^d}
\newcommand{\ff}{\mathbb{F}}

\makeatletter
\DeclareRobustCommand\widecheck[1]{{\mathpalette\@widecheck{#1}}}
\def\@widecheck#1#2{%
    \setbox\z@\hbox{\m@th$#1#2$}%
    \setbox\tw@\hbox{\m@th$#1%
       \widehat{%
          \vrule\@width\z@\@height\ht\z@
          \vrule\@height\z@\@width\wd\z@}$}%
    \dp\tw@-\ht\z@
    \@tempdima\ht\z@ \advance\@tempdima2\ht\tw@ \divide\@tempdima\thr@@
    \setbox\tw@\hbox{%
       \raise\@tempdima\hbox{\scalebox{1}[-1]{\lower\@tempdima\box
\tw@}}}%
    {\ooalign{\box\tw@ \cr \box\z@}}}
\makeatother

\stackMath
\newcommand\reallywidehat[1]{%
\savestack{\tmpbox}{\stretchto{%
  \scaleto{%
    \scalerel*[\widthof{\ensuremath{#1}}]{\kern.1pt\mathchar"0362\kern.1pt}%
    {\rule{0ex}{\textheight}}
  }{\textheight}%
}{2.4ex}}%
\stackon[-6.9pt]{#1}{\tmpbox}%
}
\definecolor{lightblue}{HTML}{2B77A4}
\colorlet{plotblue}{LightSkyBlue3!80}
\definecolor{darkred}{HTML}{9E0D0D}
\definecolor{purp}{HTML}{d603a9}
\definecolor{dartmouthgreen}{HTML}{00A64F}
\definecolor{Junglegreen}{HTML}{00A99A}
\definecolor{yellowcolour}{HTML}{f07c02}
\hypersetup{
	colorlinks=true,
	linkcolor=darkred,
	urlcolor=darkred,
	citecolor=lightblue
}
\title{An improved $L^2$ restriction theorem in finite fields}

\usepackage{fancyhdr}
\author{Jonathan M. Fraser}
\authoraddress{Jonathan M. Fraser, University of St Andrews, Scotland}
\authoremail{jmf32@st-andrews.ac.uk}
\thanks{JMF was  financially supported by a  \emph{Leverhulme Trust Research Project Grant} (RPG-2023-281),  an \emph{EPSRC Standard Grant} (EP/Y029550/1), and an \emph{EPSRC Open Fellowship} (EP/Z533440/1).}

 \author{Firdavs Rakhmonov}
\authoraddress{Firdavs Rakhmonov, University of St Andrews, Scotland}
\authoremail{fr52@st-andrews.ac.uk}
\thanks{FR was  financially supported by a  \emph{Leverhulme Trust Research Project Grant} (RPG-2023-281).}

\date{}

\begin{document}
\thispagestyle{empty}

\begin{abstract}
Mockenhaupt and Tao (Duke 2004) proved a finite field analogue of the Stein--Tomas restriction theorem, establishing a range of $q$ for which $L^q\to L^2$ restriction estimates hold for a given measure $\mu$ on a vector space over a finite field. Their result is expressed in terms of exponents that describe uniform bounds on the measure and its Fourier transform.  We generalise this result by replacing the uniform bounds on the Fourier transform with suitable $L^p$ bounds, and we show that our result improves upon the Mockenhaupt--Tao range in many cases.  We also provide a number of applications of our result, including to Sidon sets and Hamming varieties. 
\end{abstract}
\maketitle
\tableofcontents

\section{Introduction}

\subsection{The restriction problem and overview of results} \label{sec:restrictionIntro}

Given a nonzero, finite, compactly supported Borel measure $\mu$ on $\rd$, the celebrated \emph{restriction problem} asks when it is meaningful to restrict the Fourier transform of a function to the support of $\mu$. Particular cases of interest include when $\mu$ is the surface measure on sets such as the sphere, cone, or paraboloid. 

We focus on the $L^2$ theory, where the influential Stein--Tomas restriction theorem provides estimates in terms of the Fourier decay and scaling properties of $\mu$. The most general version, which we now recall, is due to  Bak--Seeger \cite{BS11}, and builds on the work of Stein \cite{stein}, Tomas \cite{Tom75}, Mockenhaupt \cite{Moc00}, Mitsis \cite{Mit02}, and others.  

\begin{thm}[Stein--Tomas]
    Let $\mu$ be a nonzero, finite, compactly supported Borel measure  on $\mathbb{R}^d$, and let $0<\alpha,\beta<d$. Suppose that for all $x\in \mathbb{R}^d$ and all $r>0$, 
    \begin{equation*}
        \mu(B(x,r))\lesssim r^{\alpha},
    \end{equation*}
    and for all $\xi\in \mathbb{R}^d$,
    \begin{equation*}
    |\widehat{\mu}(\xi)|\lesssim |\xi|^{-{\beta}/{2}}.
    \end{equation*}
    Then 
    \begin{equation}
    \label{extension estimate}
        \lVert \widehat{f\mu}\rVert_{L^q(\mathbb{R}^d)}\lesssim_{q,\alpha,\beta}\lVert f\rVert_{L^2(\mu)}
    \end{equation}
    holds for all functions $f\in L^2(\mu)$ and all $q\geq 2+\frac{4(d-\alpha)}{\beta}$.
\end{thm}

The estimate \eqref{extension estimate} is the $L^2 \to L^q$ extension estimate, which, by duality, is equivalent to the $L^{q'} \to L^2$ restriction estimate 
\begin{equation*}
\| \widehat f \|_{L^2(\mu)} \lesssim \|f \|_{L^{q'}(\mathbb{R}^d)},
\end{equation*}
which says  that we may meaningfully interpret the restriction of the Fourier transform of an $L^{q'}$ function in the $L^2$ sense.  

Mockenhaupt and Tao \cite{moctao} proved a finite field analogue of the Stein--Tomas restriction theorem; see Theorem \ref{moctaoresult} below. Analogous to the Stein--Tomas result, their theorem provides a range based on \emph{uniform} bounds for the Fourier transform of the measure. Our main result, Theorem \ref{main}, generalises the Mockenhaupt--Tao result by allowing the range to depend on suitable $L^p$ averages of the Fourier transform.

We  show that our result improves upon the Mockenhaupt--Tao range in many cases and provide several applications of our estimate, including to restriction theory for Sidon sets (Corollary \ref{sidonrestriction}) and Hamming varieties  (Corollary \ref{hammingresult}).

In the above and throughout, the notation $A\lesssim B$ signifies that $A\leq c B$ for some constant $c>0$ depending only on the ambient spatial dimension $d$.  Similarly, we write $A \gtrsim B$ to mean $B \lesssim A$, and $A \approx B$ if both $A\lesssim B$ and $A \gtrsim B$ hold. We will use subscripts to indicate that the implicit constants depend on other parameters, such as $q$, $\alpha$, and $\beta$ in \eqref{extension estimate}. The implicit constants will never depend on the size of the base field $\mathbb{F}$.   We also  write $A \gg B$ to denote the negation of $A \lesssim B$.

We also write $q'$ for the H\"older conjugate of $q \in [1,\infty]$, i.e., the unique $q'\in [1,\infty]$ satisfying $\frac{1}{q}+\frac{1}{q'}=1$. Additionally, we use $|X|$ to denote the cardinality of a finite set $X$.

\subsection{Discrete Fourier analysis and the restriction problem over finite fields}

In this  section, we first provide an overview of discrete Fourier analysis in vector spaces over finite fields and refer the reader to \cite{nied} for a more general background on finite fields. 

Let $\ff$ be a finite field, and let $\ff^d$ be the $d$-dimensional vector space over $\ff$, where $d \geq 1$ is an integer. The \emph{Fourier transform} and \emph{inverse Fourier transform} of a function $f : \ff^d \to \mathbb{C}$ are the functions $\widehat f :  \ff^d \to \mathbb{C}$ and $f^{\lor}:\ff^d\to \mathbb{C}$, defined by
\begin{equation*}
\label{def of FT}
\widehat f(\xi) \coloneqq   \sum_{x \in \ff^d} f(x)\chi(-\xi  \cdot x),
\end{equation*}
\begin{equation*}
f^{\lor} (\xi) \coloneqq \sum_{x \in \ff^d} f(x)\chi(\xi  \cdot x),
\end{equation*}
where $\chi : \ff \to S^1 \subseteq \mathbb{C}$ is a nontrivial additive character. The specific choice of $\chi$ does not play an important role in what follows. Here, $\xi \cdot x$ denotes the standard dot product on $\ff^d$, which takes values in $\ff$.

The following identity shows that one can interchange the Fourier and inverse Fourier transforms: for every $x\in \ff^d$, we have
\begin{equation} 
\label{inversion}
(\widehat{f})^{\lor}(x) = \widehat{{f^{\lor}}}(x)   = |\ff|^d f(x).
\end{equation}
The following result, known as \emph{Parseval's theorem}, will be very useful throughout the paper:
\begin{equation}
\label{parseval}
 \sum_{\xi \in \ff^d} \widehat f(\xi) \overline{\widehat g(\xi)}  = |\ff|^d \sum_{x \in \ff^d}  f(x) \, \overline{g(x)}.
\end{equation}
The special case of \eqref{parseval} with $f=g$ is called \emph{Plancherel's formula}:
\begin{equation}
\label{plancherel}
 \sum_{\xi \in \ff^d} |\widehat f(\xi)|^2  = |\ff|^d \sum_{x \in \ff^d}  |f(x)|^2.
\end{equation}
For $E\subseteq \ff^d$, we define $E(x)$ to be the indicator function of $E$, that is, 
\begin{equation*}
E(x) \coloneqq
\begin{cases}
0, & \text{if }x\notin E, \\
1, & \text{if }x\in E.
\end{cases} 
\end{equation*}
For functions $f,g:\ff^d\to \mathbb{C}$, we define the \emph{convolution} as follows:
\begin{equation*}
    f * g (x) \coloneqq \sum_{y \in \ff^d} f(y)g(x-y),
\end{equation*}
and it is not difficult to prove the \emph{convolution law}: for every $\xi \in \ff^d$, we have:
\begin{equation} 
\label{convolution law}
\widehat {f* g}(\xi) = \widehat f(\xi) \widehat g(\xi).
\end{equation}
We are interested in the restriction problem in $\fd$.  A \emph{probability measure} $\mu$ on $\fd$ is simply a non-negative function that sums to 1.  For $E \subseteq \fd$,   the \emph{surface measure} on $E$ is the uniform probability measure on $E$, namely, 
\[
\mu(x) = \frac{E(x)}{|E|}.
\]
Therefore, by the definition of the Fourier transform, we have
\[
\widehat \mu(\xi) =  \frac{1}{|E|} \sum_{x \in E} \chi(-\xi \cdot x), 
\]
where $\mu$ is the surface measure on $E$. In the continuous setting, one often studies the decay of the Fourier transform of a function or measure at infinity. However, in the finite field setting, there is no notion of \emph{infinity}; instead, one seeks uniform estimates over $\ff^d$. When seeking uniform bounds for the Fourier transform of a probability measure  $\mu$ on $ \fd$, one always has the trivial estimate
\begin{equation} \label{trivial}
|\widehat \mu (\xi) | \leq    1 
\end{equation}
for all $\xi \in \ff^d$, and this bound is always attained since $\widehat \mu (0)=1$ .  On the other hand, by Plancherel's Theorem \eqref{plancherel}, if $\mu$ is the surface measure on a set $E$, then 
\begin{equation*}
 |\ff|^{d} |E|^{-1} =   |\ff|^d \sum_{x \in \ff^d} | \mu(x)|^2 = \sum_{\xi \in \ff^d} |\widehat \mu(\xi)|^2 \leq 1+ |\ff|^d \sup_{\xi \neq 0} |\widehat \mu(\xi)|^2,
\end{equation*}
and so the best one can reasonably hope for is 
 \begin{equation} \label{iosevichassumption}
|\widehat \mu(\xi) | \lesssim  |E|^{-\frac{1}{2}}
\end{equation}
for all $\xi \in \ff^d\setminus \{0\}$. This estimate certainly cannot be improved upon if $|E| = o(|\ff|^d)$.   In \cite{iosevich}, sets satisfying \eqref{iosevichassumption} were called \emph{Salem sets}, i.e., sets for which the surface measure has optimal uniform bounds on its Fourier transform.

Mockenhaupt and Tao \cite{moctao} established the following  finite field analogue of the Stein--Tomas restriction theorem, also see \cite{chen}. We state only the extension estimate here, as the corresponding restriction estimate follows by duality. For a function $g: \fd \to \mathbb{C}$, we define
\begin{equation*}
\|g\|_{L^q(\fd)}  \coloneqq \Bigg(\sum_{x \in \fd} |g(x)|^q \Bigg)^{\frac{1}{q}}
\quad \textup{and} \quad \|g\|_{L^q(\mu)}  \coloneqq \Bigg(\sum_{x \in \fd} |g(x)|^q \mu(x) \Bigg)^{\frac{1}{q}}.
\end{equation*}

\begin{thm}[Mockenhaupt--Tao] \label{moctaoresult}
Let $\mu$ be a probability measure on $\fd$, and let $0<\alpha,\beta_{\infty}<d$. Suppose that $\mu(x) \lesssim  |\ff|^{-\alpha}$
for all $x \in \fd$ and $|\widehat \mu (\xi) |\lesssim |\ff|^{-{\beta_{\infty}}/{2}}$ for all $\xi \in \fd \setminus \{0\}$.  Then
\[
    \|\widehat{f\mu}\|_{L^q(\fd)}\lesssim_{q,\alpha,\beta_{\infty}} \|f\|_{L^2(\mu)}
\]
holds for all functions $f:\fd\to \mathbb{C}$ and all $q \geq 2+\frac{4(d-\alpha)}{\beta_\infty}$.
\end{thm}

\section{Main results}

In \cite{fraserfinite}, a new approach for capturing the Fourier decay in the setting of finite fields was introduced.  Namely, for a given $p \in [1,\infty]$ and $s\in [0,1]$, we say that a set $E \subseteq \fd$ is a \emph{$(p,s)$-Salem set} if the surface measure $\mu$ on $E$ satisfies
\[
\| \widehat \mu \| _p \lesssim_{p,s} |E|^{-s},
\]
where
\begin{equation*}
\| \widehat \mu \| _p \coloneqq \Bigg( |\ff|^{-d} \sum_{\xi \neq 0} | \widehat \mu (\xi)|^p \Bigg)^{\frac{1}{p}}
\end{equation*}
for $p\in [1,\infty)$, and
\[
\| \widehat \mu \| _\infty \coloneqq  \sup_{\xi \neq 0} |\widehat \mu(\xi)|.
\]
In \cite{fraserfinite}, the notation was slightly different (for example, the Fourier transform was normalized), but we adopt the  notation above to facilitate comparison with the literature, particularly \cite{moctao}. 

It is straightforward to see that  $E$ is a Salem set if and only if it is an $(\infty, \frac{1}{2})$-Salem set. Moreover, every set is a $(p,\frac{1}{p})$-Salem set for all $p \in [2,\infty]$.  Furthermore, if $|E| = o(|\mathbb{F}|^d)$, then the best one can hope for is that $E$ is a $(p, \tfrac{1}{2})$-Salem set for any given $p$.

The idea is that by considering a continuum of $L^p$ averages of the Fourier transform, one may obtain more nuanced information about the Fourier analytic behaviour of $\mu$, which in turn can lead to stronger applications.  Indeed, this is the case for Fourier restriction, and we can prove the following generalisation of the Mockenhaupt--Tao theorem. The following result serves as a finite field analogue of a restriction theorem of Carnovale--Fraser--de Orellana \cite{ana} and is the main result of the paper.

\begin{thm} 
\label{main}
Let $p\in [1,\infty]$, and let $\mu$ be a probability measure on $\fd$, with $0<\alpha,\beta_p<d$. If $\mu(x) \lesssim  |\ff|^{-\alpha}$ for all $x \in \fd$ and $\|\widehat \mu \|_p \lesssim |\ff|^{-\beta_p/2}$ for some $\beta_p \geq \frac{2d}{p}$, then
\begin{equation}
\label{norm inequality}
    \|\widehat{f\mu}\|_{L^q(\fd)}\lesssim \|f\|_{L^2(\mu)}
\end{equation}
holds for all functions $f:\fd\to \mathbb{C}$ and all 
\[
q \geq  2+\frac{(4p-4)(d-\alpha)}{p\beta_p-2\alpha} .
\]
\end{thm}

We prove Theorem \ref{main} in Section \ref{proofmain}.  First, we observe that by setting $p=\infty$, we recover the Mockenhaupt--Tao result, since the seemingly additional condition $\beta_p \geq \frac{2d}{p}$ becomes $\beta_\infty \geq 0$, which is trivially satisfied.  Moreover, we obtain a strictly better range whenever there exists $p \in [2,\infty)$ such that the assumptions hold with
\[
\beta_p > \beta_\infty+\frac{2\alpha-\beta_\infty}{p} \qquad \text{and} \qquad \beta_p \geq \frac{2d}{p},
\]
where $\beta_\infty$ is the decay exponent from the Mockenhaupt--Tao setting.

For ease of comparison with \cite{fraserfinite} and the examples therein, we state a corollary using the terminology from \cite{fraserfinite}. The only additional  observation needed is that if $|E| \approx |\ff|^\alpha$ and $E$  is a $(p,s)$-Salem set, then
\[
\mu(x) \lesssim  |\ff|^{-\alpha}
\]
for all $x \in \fd$, and 
\[
\|\widehat \mu \|_p \lesssim |\ff|^{-\alpha s},
\]
where $\mu$ denotes the surface measure on $E$.

\begin{cor}  
\label{cormain}
Let $0<\alpha<d$, and let $E \subseteq \fd$ be such that $|E| \approx |\ff|^\alpha$. Suppose that $E$ is a $(p,s)$-Salem set with  $s \geq \frac{d}{p\alpha}$.  Then, for $\mu$ the surface measure on $E$,
\[
    \|\widehat{f\mu}\|_{L^q(\fd)}\lesssim \|f\|_{L^2(\mu)}
\]
holds for all functions $f:\fd\to \mathbb{C}$, provided that 
\[
q \geq  2+\frac{(2p-2)(d-\alpha)}{ \alpha p s-\alpha} .
\]
In particular, this improves upon the Mockenhaupt--Tao range when
\[
s>s_\infty+\frac{1-s_\infty}{p},
\]
where $s_\infty$ is chosen optimally so that $E$ is $(\infty, s_\infty)$-Salem set.
\end{cor}

Note that from \cite[Proposition 2.1]{fraserfinite}, we know that \emph{all} sets $E$ are $(p,s)$-Salem set for 
\[
s = s_\infty+\frac{1-2s_\infty}{p},
\]
and so the threshold for improving upon the Mockenhaupt--Tao range is quite reasonable. We will see later that this improvement occurs in several examples. 
 
In the other direction, we note the following negative result, framed in terms of our  $L^p$ approach.  

\begin{prop} 
\label{converse}
Let $\mu$ be a probability measure on $\fd$   such that $\|\widehat \mu \|_p \gg |\ff|^{-\beta_p/2}$ for some $\beta_p \leq \frac{2d}{p}$. Then,
\[
    \|\widehat{f\mu}\|_{L^p(\fd)} \gg 1
\]
for  $f=E(\cdot)-\delta_0$, where $\delta_0$ denotes the Dirac delta measure at $0\in \mathbb{F}_q^d$. In particular, there are no $L^{q_0} \to L^q$ extension estimates for any $q_0$ with $q \leq p$. 
\end{prop}

\begin{proof}[Proof of Proposition \ref{converse}]
Setting $f=E(\cdot)-\delta_0$, we immediately obtain
\[
   \|\widehat{f\mu}\|_{L^p(\fd)} = |\ff|^{\frac{d}{p}} \| \widehat \mu \|_p \gg |\ff|^{\frac{d}{p}}|\ff|^{-\frac{\beta_p}{2}} \geq 1,
\]
as required.
\end{proof}

Proposition \ref{converse} generalises the observation by Mockenhaupt--Tao that there are no $L^{q_0} \to L^q$ extension estimates for any $q_0$ with $q < \frac{2d}{\alpha}$ when $\mu$ is the surface measure on $E$ with $E\approx q^\alpha$ and  $\alpha<d$ \cite[(9)]{moctao}.  In particular, we get
\[
q< \sup\left\{p : \text{$E$ fails to be $\Big(p, \frac{d}{\alpha p}\Big)$-Salem}\right\}.
\]
Since no such $E$ can be $(p,\frac{1}{2}+\varepsilon)$-Salem we recover the range $q <\frac{2d}{\alpha}$, but we will often get a better bound.

 \section{Proof of Theorem \ref{main}} \label{proofmain}

Our proof follows the general strategy of the Mockenhaupt--Tao argument \cite{moctao}, but we replace their $L^1 \to L^\infty$ estimate with an $L^{(2p)'} \to L^{2p}$ estimate, which we obtain via $\|\widehat \mu \|_p$. We  found the exposition in \cite{chen} very helpful for following the argument. 

Instead of proving \eqref{norm inequality}, we will prove the restriction version, namely,
\[
\| \widehat{f} \|_{L^2(\mu)} \lesssim \|f \|_{L^{q'}(\ff^d)}.
\]  
Let $q \geq 2$ be arbitrary for now.  Then, we obtain:
\begin{align*}
\| \widehat{f} \|_{L^2(\mu)}^2 &=  \sum_{\xi \in \fd} |\widehat{f}(\xi)|^2 \mu(\xi) \\ 
& = \sum_{\xi \in \fd} \widehat{f}(\xi) \overline{\widehat{f}(\xi)\mu(\xi)} \\
&=|\ff|^{-d}\sum_{\xi \in \fd} \widehat{f}(\xi)\overline{\widehat{f}(\xi) \widehat{\mu^{\lor}}(\xi)} \qquad \text{(by inversion \eqref{inversion})}\\
&=|\ff|^{-d}\sum_{\xi \in \fd} \widehat{f}(\xi)\overline{\widehat{f*\mu^{\lor}}(\xi)} \qquad \text{(by the convolution law  \eqref{convolution law})}\\
&=\sum_{x \in \fd} f(x)\overline{f*\mu^{\lor}(x)}\qquad \text{(by Parseval \eqref{parseval})}\\
&\leq \lVert f \rVert_{L^{q'}(\ff^d)} \lVert f*\mu^{\lor} \rVert_{L^{q}(\ff^d)} \qquad \text{(by Hölder's inequality)}. 
\end{align*}
Having this estimate in mind, in order to prove the main result, it suffices to establish
\begin{equation} \label{want}
\|f \ast \mu^{\lor}\|_{L^{q}(\fd)} \lesssim  \| f \|_{L^{q'}(\fd)}
\end{equation}
for the appropriate range of $q$. Let $\delta_0$ denote the Dirac delta measure at $0\in \mathbb{F}_q^d$.  Then, by Minkowski's inequality, 
\[
\|f \ast \mu^{\lor}\|_{L^{q}(\fd)} \leq \|f \ast (\mu^{\lor}-  \delta_0)\|_{L^{q}(\fd)} +  \| f \ast   \delta_0\|_{L^{q}(\fd)}. 
\]
Since $f*\delta_0=f$,  for the latter term we easily obtain
\begin{equation} \label{easy}
 \| f \ast   \delta_0\|_{L^{q}(\fd)}  =   \| f \|_{L^{q}(\fd)} \leq \| f \|_{L^{q'}(\fd)},
\end{equation}
where the last inequality follows from the fact that $1\leq q'\leq 2 \leq q$.  For the other term, we obtain the desired $L^{q'} \to L^q$ bound by interpolation. Since $1+\frac{1}{2p}=\frac{2p-1}{2p}+\frac{1}{p}$,  applying Young's convolution inequality gives
\begin{equation*}
\|f \ast (\mu^{\lor}-  \delta_0)\|_{L^{2p}(\fd)} \leq \|f  \|_{L^{\frac{2p}{2p-1}}(\fd)}\| \mu^{\lor}-\delta_0 \|_{L^{p}(\fd)}. 
\end{equation*}
It is not difficult to verify that $\| \mu^{\lor}-\delta_0 \|_{L^{p}(\fd)}=|\ff|^{\frac{d}{p}} \|\widehat{\mu} \|_p$, and since $\|\widehat{\mu} \|_p\lesssim |\mathbb{F}|^{-\frac{\beta_p}{2}}$, we obtain:
\begin{equation} \label{i1}
\|f \ast (\mu^{\lor}-  \delta_0)\|_{L^{2p}(\fd)} \leq |\mathbb{F}|^{\frac{d}{p}-\frac{\beta_p}{2}}\|f\|_{L^{\frac{2p}{2p-1}}(\fd)} .
\end{equation}
On the other hand, we have:
\begin{align*}
&  \hspace{-5mm} \|f \ast (\mu^{\lor}-  \delta_0)\|_{L^{2}(\fd)}^2 \\
&=  |\ff|^{-d} \sum_{\xi \in \fd} |\widehat f (\xi)|^2  | |\widehat{(\mu^{\lor}- \delta_0)}(\xi)|^2  \qquad \text{(by Plancherel \eqref{plancherel} and the convolution law)}\\
&=|\ff|^{-d} \sum_{\xi \in \fd} |\widehat f (\xi)|^2  | |\ff|^d\mu(\xi)- \widehat{\delta_0}(\xi)|^2  \qquad \text{(by inversion \eqref{inversion})}\\
&\leq \sup_{\xi\in \mathbb{F}^d}||\mathbb{F}|^d\mu(\xi)-1|^2\sum_{x \in \mathbb{F}^d}|f(x)|^2 \qquad \text{(by Plancherel \eqref{plancherel})} \\
&\lesssim|\ff|^{2d-2\alpha}    \|f  \|_{L^{2}(\fd)}^2. 
\end{align*}
Therefore, we have 
\begin{equation}
\label{i2}
    \|f \ast (\mu^{\lor}-  \delta_0)\|_{L^{2}(\fd)}\lesssim |\ff|^{d-\alpha}    \|f  \|_{L^{2}(\fd)}. 
\end{equation}
Applying the Riesz--Thorin interpolation theorem with \eqref{i1} and \eqref{i2}, and parameter $\lambda \in [0,1]$, and using \eqref{easy}, we get
\[
\|f \ast (\mu^{\lor}-  \delta_0)\|_{L^{q(\lambda)}(\fd)} \lesssim  |\ff|^{(d/p-\beta_p/2)(1-\lambda)+(d-\alpha) \lambda }  \| f \|_{L^{q'(\lambda)}(\fd)},
\]
where
\[
q(\lambda) = \frac{2p}{1-\lambda+\lambda p}.
\]
In particular, we obtain \eqref{want} provided that
\[ 
(d/p-\beta_p/2)(1-\lambda)+(d-\alpha) \lambda \leq 0.
\]
Therefore, the optimal choice of $\lambda$ is 
\[
\lambda = \frac{\beta_p/2-d/p}{\beta_p/2 -d/p+d-\alpha} \in [0,1],
\]
noting that we are assuming $\beta_p  \geq \frac{2d}{p}$. For this choice of $\lambda$ we obtain the $L^{2} \to L^{q_0}$ extension estimate, where
\[
q_0 = \frac{p(4d-4\alpha+2\beta_p)-4d}{p\beta_p-2\alpha}=2+\frac{(4p-4)(d-\alpha)}{p\beta_p-2\alpha},  
\]
which is the required endpoint estimate. For larger $q$, the result follows by monotonicity.

\section{Simple examples}

In this section, we investigate a number of simple examples with an emphasis on when our extension estimate improves upon the Mockenhaupt--Tao estimate. 

\subsection{Products of spheres}

Let $k \geq 2$ and $m \geq 1$ be integers, and let $E\coloneqq (S_1^{k-1})^m = S_1^{k-1} \times \cdots \times S_1^{k-1} \subseteq \ff^{km}$ be the $m$-fold product of the sphere of radius 1 (or any non-zero radius) in $\ff^k$.  Using the fact that $S_1^{k-1}$ is a Salem set with $S_1^{k-1} \approx |\ff|^{k-1}$ and applying \cite[Theorems 3.2 and 3.5]{fraserfinite}, we obtain the following.
\begin{prop} \label{sphereproduct}
Let $p \geq 2$. The set $E$ is $(p,s)$-Salem if and only if 
\[
s \leq s_p:=\min\left\{ \frac{2k(m-1)+p(k-1)}{2mp(k-1)}, \frac{1}{2}\right\}.
\]
\end{prop}
In fact, \cite[Theorems 3.2 and 3.5]{fraserfinite} give estimates for the product of two sets, and so one establishes the above formula inductively, using $(S_1^{k-1})^m = (S_1^{k-1})^{m-1} \times S_1^{k-1}$ as the inductive step, and the sphere $S_1^{k-1}$ itself as the base case. We omit the details. 

Applying Corollary \ref{cormain} together with the formula for $s_p$ in Proposition \ref{sphereproduct} allows us to obtain an extension estimate  for $\mu$, the surface measure on $E$, for
\[
q \geq  \frac{p(2d-2\alpha+2\alpha s_p)-2d}{ \alpha s_p p-\alpha} 
\]
whenever  $s_p \geq \frac{d}{p\alpha}$. Here, $d=mk$ is the ambient spatial dimension, and $\alpha=(k-1)m$ is the `dimension' of $E$.  The formula for $s_p$ has a phase transition at $p=\frac{2k}{k-1}$, and it is easy to see that the extension bound is optimised either at this  choice of $p$ or at $p = \infty$, with the $p=\infty$ case corresponding to the Mockenhaupt--Tao bound.   Moreover, for $p=\frac{2k}{k-1}$, $s_p=\frac{1}{2} =\frac{d}{p\alpha}$, so the bound is indeed valid at both extremes. Finally, since $s_p <\frac{d}{p\alpha}$ for $p < \frac{2k}{k-1}$, we get a range of $q$ where the extension estimate fails by appealing to Proposition \ref{converse}.  Pulling this all together, we get the following result.
\begin{cor}
    For $\mu$, the surface measure on $E\coloneqq(S_1^{k-1})^m\subseteq\mathbb{F}^{km}$,
\[
    \|\widehat{f\mu}\|_{L^q(\mathbb{F}^{km})}\lesssim \|f\|_{L^2(\mu)}
\]
holds for all functions $f:\mathbb{F}^{km}\to \mathbb{C}$, provided that 
\[
q \geq  2+   \frac{\min\{2k+2,4m\}}{k-1}.
\]
Moreover, the above extension estimate fails for
\[
q < 2+\frac{2}{k-1}.
\]
\end{cor}
The Mockenhaupt--Tao result gives the desired extension estimate for
\[
q \geq 2+\frac{4m}{k-1},
\]
so our estimate is strictly better whenever $k < 2m-1.$

\subsection{The sphere of radius zero}

Here, we consider a concrete example involving the sphere of radius zero. Let $E \coloneqq S_0^2 \times S_1^2 \subseteq \ff^6$, where $S_0^2 \subseteq \ff^3$ is the sphere of radius zero and $S_1^2 \subseteq \ff^3$ is the sphere of radius 1 (or any non-zero radius). Let   $\mu$ denote the surface measure on $E$.  Then, by \cite[Theorems 3.1,  3.2 and 3.5]{fraserfinite}, the set $E$ is $(p,s)$-Salem if and only if
\[
s \leq \min\left\{\frac{1}{8}+\frac{1}{p}, \frac{3}{4p}+\frac{1}{4}\right\}.
\]
This threshold exceeds $\frac{d}{p\alpha} = \frac{3}{2p}$ for $p \geq 4$ and combining this with Corollary \ref{cormain}, we obtain
\[
    \|\widehat{f\mu}\|_{L^q(\mathbb{F}^6)}\lesssim \|f\|_{L^2(\mu)}
\]
 holds for all functions $f:\mathbb{F}^6 \to \mathbb{C}$ and all $q \geq 8$. This range is achieved uniquely at $p=4$, while the Mockenhaupt--Tao range (corresponding to $p=\infty$) is $q \geq 10$. Moreover, Proposition \ref{converse} shows that the above $L^2\to L^q$ extension estimate fails for $q< 4$.

\subsection{Cutoff cylinders}

First, consider the cylinder $\ff^n \times S_1^k \subseteq \ff^{n+k+1}$.  For these examples, our results do not improve upon those of Mockenhaupt--Tao. However, perhaps surprisingly, we do obtain an improvement for the \emph{cutoff cylinders}
\[
E \coloneqq(\ff^n \setminus \ff^m) \times S_1^k \subseteq \ff^{n+k+1}
\]
for integers $n >m\geq 1$ and $k >2(n-m)$.  Here, we view $\ff^m$ as a subset of $\ff^n$. Applying  \cite[Theorem 3.2 and Proposition 3.8]{fraserfinite}, we find that $E$ is $(p,s)$-Salem if and only if
\begin{equation} \label{cutformula}
s \leq \min\left\{\frac{n/p+k/2}{n+k},\frac{n-m+m/p+(k+1)/p }{n+k}\right\}.
\end{equation}
We may apply Theorem \ref{main} for $p \geq \frac{2(k+1)}{k}$, and the optimal result is obtained at
\[
p=\frac{2(k+1+m-n)}{k+2(m-n)} < \infty,
\]
which corresponds to the phase transition in the right-hand side of \eqref{cutformula}.  For this value of $p$, we obtain
\[
    \|\widehat{f\mu}\|_{L^q(\ff^{n+k+1})}\lesssim \|f\|_{L^2(\mu)}
\]
for all functions $f:\mathbb{F}^{n+k+1}\to \mathbb{C}$, provided that 
 \[
 q \geq \frac{4+k(2n-2m+4)}{k(n-m+1)}.
 \]
 On the other hand, the Mockenhaupt--Tao result yields the range
 \[
 q \geq \frac{2+2(n-m)}{n-m},
 \]
 which is strictly weaker, recalling that $k >2(n-m)$.

\section{Applications}

\subsection{Large Sidon sets}

Here we prove that large Sidon sets satisfy non-trivial restriction estimates, whereas smaller Sidon sets need not. A \emph{Sidon set} $E \subseteq \fd$ is a set for which the equation $a+b=c+d$ implies $\{a,b\} = \{c,d\}$ for all $a,b,c,d \in E$.  That is, $E$ generates the largest possible number of distinct sums.  As a consequence, if $E$ is Sidon, then $|E| \lesssim q^{\frac{d}{2}}$ always holds, but it is easy to construct  Sidon sets with $|E| \approx q^{\frac{d}{2}}$.  We show that non-trivial restriction estimates always hold for such large Sidon sets.

We are unaware of any results providing non-trivial uniform bounds for the Fourier transform of the surface measure on a general Sidon set, and so, without more information about $E$, we are unable to apply the Mockenhaupt--Tao result in this case.

\begin{cor} \label{sidonrestriction}
Let $E \subseteq \fd$ be a Sidon set with $|E| \approx q^{\frac{d}{2}}$, and let $\mu$ be the surface measure on $E$.  Then 
\[
    \|\widehat{f\mu}\|_{L^8(\fd)}\lesssim \|f\|_{L^2(\mu)}
\]
 holds for all functions $f:\mathbb{F}^d \to \mathbb{C}$.
\end{cor}

\begin{proof}[Proof of Corollary \ref{sidonrestriction}]
It follows from \cite[Corollary 7.3]{fraserfinite} that $E$ is $(p,\frac{2}{p})$-Salem for all $p \geq 4$, which means that the condition $s \geq \frac{d}{p\alpha}$ required to apply our main result is satisfied for all such $p$.  Then Corollary \ref{cormain} gives that
\[
    \|\widehat{f\mu}\|_{L^q(\fd)}\lesssim \|f\|_{L^2(\mu)}
\]
 holds for all functions $f:\mathbb{F}^d \to \mathbb{C}$ and all 
 \[
 q \geq \frac{p(2d-2(d/2)+2(d/2)(2/p))-2d}{(d/2)(2/p)p-d/2} = \frac{pd}{d/2} = 2p,
 \]
 and this range is optimised by choosing $p=4$ giving $q \geq 8$. 
\end{proof}

At first glance, the fact that Corollary \ref{sidonrestriction} only applies to Sidon sets with $|E| \approx q^{\frac{d}{2}}$ might seem unsatisfying.  However, this condition is close to optimal.  Indeed, suppose $E \subseteq \ff^{d-1}$ is a Sidon set with $|E| \approx q^{\frac{d-1}{2}}$, and embed it as a subset of $\ff^d$.  Then \cite[Corollary 3.6]{fraserfinite} shows that $E$ is \emph{not} $(p,s)$-Salem for $s > \min\big\{\frac{2}{p},\frac{1}{2}\big\}$. Proposition \ref{converse} then ensures that the extension estimate fails for \emph{all} $q <\infty$. In particular, letting $d \to \infty$ gives a family of Sidon sets with `dimension' $\sim \frac{d}{2}$ for which there are no non-trivial extension estimates.

\subsection{Hamming varieties}

In this section, we consider an interesting and explicit example of a variety in $\mathbb{F}^d$ which does not admit good uniform bounds for the Fourier transform away from the origin, but does exhibit better bounds on average.  This type of behaviour suggests that our approach should yield  a better $L^2\to L^q$ extension estimate for the surface measure on this variety than  the Mockenhaupt--Tao result.  Indeed, we will see this is the case.

For each $j\in \mathbb{F}^*$, the Hamming variety $H_j$ in $\mathbb{F}^d$ is defined by 
\begin{equation*}
    H_j =\Big\{x=(x_1,\dots,x_d)\in \mathbb{F}^d: \prod_{k=1}^dx_k=j\Big\}.
\end{equation*}

Since $j\neq 0$, it is not difficult to verify that $|H_j|=(|\mathbb{F}|-1)^{d-1}\approx |\mathbb{F}|^{d-1}$. Let $\mu_j$ denote the surface measure on the Hamming variety $H_j$. The following result provides an explicit expression for $\widehat{\mu_j}$. The proof is fairly straightforward and can be found in \cite{hamming}.

\begin{lma}
\label{hamming decay}
    For each $j\in \mathbb{F}^*$, let $\mu_j$ be the surface measure on the Hamming variety $H_j$. For each $m\in \mathbb{F}^d$, let $\ell_m$ denote the number of coordinates of $m$ which are equal to zero. Then: 
    \begin{equation*}
        \widehat{\mu_j}(m)=(-1)^{d-\ell_m}(|\mathbb{F}|-1)^{-(d-\ell_m)} \quad \textit{if} \quad 1\leq \ell_m\leq d.
    \end{equation*}
    Moreover, if $\ell_m=0$, then $|\widehat{\mu_j}(m)|\lesssim |\mathbb{F}|^{-\frac{d-1}{2}}$.
\end{lma}

Using these estimates, we can explicitly describe how the $L^p$ averages of $ \widehat{\mu_j}(m)$ behave.  This is a new example and is of interest in its own right in the context of the programme introduced in \cite{fraserfinite}.

\begin{prop} \label{hammingformula}
    The Hamming variety $H_j$ is $(p,s)$-Salem if and only if 
    \[
    s \leq \min\left\{\frac{1}{d-1}+\frac{1}{p}, \frac{1}{2}\right\}.
    \]
\end{prop}
\begin{proof}[Proof of Proposition \ref{hammingformula}]
    We first compute $\lVert \widehat{\mu_j}\rVert_p$. By definition, and then using Lemma \ref{hamming decay},
\begin{align*}
    \lVert \widehat{\mu_j}\rVert_p^p&=|\mathbb{F}|^{-d}\sum_{m\neq 0}|\widehat{\mu_j}(m)|^p\\
    &=|\mathbb{F}|^{-d}\sum_{k=1}^{d-1}\sum_{\substack{m\neq 0 \\ \ell_m=k}}|\widehat{\mu_j}(m)|^p+|\mathbb{F}|^{-d}\sum_{\substack{m\neq 0 \\ \ell_m=0}}|\widehat{\mu_j}(m)|^p\\
    &\lesssim |\mathbb{F}|^{-d}\sum_{k=1}^{d-1}\binom{d}{k}(|\mathbb{F}|-1)^{d-k}(|\mathbb{F}|-1)^{-p(d-k)}+|\mathbb{F}|^{-d}(|\mathbb{F}|-1)^d|\mathbb{F}|^{-\frac{p(d-1)}{2}}\\
    & \approx |\mathbb{F}|^{-d} \sum_{k=1}^{d-1}|\mathbb{F}|^{d-k+pk-pd}+|\mathbb{F}|^{-\frac{p(d-1)}{2}}\\
    & \approx |\mathbb{F}|^{-p-d+1}+|\mathbb{F}|^{-\frac{p(d-1)}{2}}\\
    &\lesssim |\mathbb{F}|^{-\min \left\{p+d-1, \frac{p(d-1)}{2}\right\}}.
\end{align*}
Therefore,
\begin{equation*}
    \lVert \widehat{\mu_j}\rVert_p \lesssim |\mathbb{F}|^{-\min\{1+\frac{d-1}{p}, \frac{d-1}{2}\}} \approx |H_j|^{-\min\{\frac{1}{d-1}+\frac{1}{p}, \frac{1}{2}\}},
\end{equation*}
since $|H_j|\approx |\mathbb{F}|^{d-1}$. This shows that the Hamming variety $H_j$ is  $(p,s)$-Salem for all
    \[
    s \leq \min\left\{\frac{1}{d-1}+\frac{1}{p}, \frac{1}{2}\right\}.
    \]
    For the other direction, following the argument above but dropping the second term (corresponding to the case when $\ell_m=0$) and noting that Lemma \ref{hamming decay} provides a precise formula in the other cases, we see that
    \begin{equation*}
    \lVert \widehat{\mu_j}\rVert_p \gtrsim  |H_j|^{-\left(\frac{1}{d-1}+\frac{1}{p}\right)}.
\end{equation*}
This proves that  $H_j$ is not  $(p,s)$-Salem for all
    \[
    s >  \frac{1}{d-1}+\frac{1}{p}, 
    \]
    and since $|H_j|\approx |\mathbb{F}|^{d-1}$, it  also cannot be $(p,s)$-Salem for $s>\frac{1}{2}$, and the result follows.
\end{proof}

Now we   apply Corollary \ref{cormain} to obtain an $L^2\to L^q$ extension estimate.

\begin{cor} \label{hammingresult}
Let $H_j$ be a Hamming variety in $\mathbb{F}^d$, and let $\mu_j$ be the surface measure on $H_j$. Then 
\begin{equation*}
    \lVert \widehat{f\mu_j}\rVert_{L^q(\mathbb{F}^d)}\lesssim \lVert f\rVert_{L^2(\mu)}
\end{equation*}
holds for all functions $f:\mathbb{F}^d\to \mathbb{C}$, provided that $q\geq \frac{3d-1}{d-1}$.
\end{cor}

\begin{proof}
  Directly from Corollary \ref{cormain} and Proposition \ref{hammingformula}, we obtain that   if $\min\big\{\frac{1}{d-1}+\frac{1}{p},\frac{1}{2}\big\}\geq \frac{d}{p(d-1)}$, then 
\begin{equation*}
    \lVert \widehat{f\mu_j}\rVert_{L^q(\mathbb{F}^d)}\lesssim \lVert f\rVert_{L^2(\mu)}
\end{equation*}
holds for all functions $f:\mathbb{F}^d \to \mathbb{C}$, and all  $$q\geq 2 + \frac{(2p-2)}{p(d-1)\min\big\{\frac{1}{d-1}+\frac{1}{p},\frac{1}{2}\big\}-(d-1)}.$$
With some computation, one can show that the optimal choice is $p=\frac{2(d-1)}{d-3}$, and the result follows.
\end{proof}

It is important to note that the Mockenhaupt--Tao result (Theorem ~\ref{moctaoresult}) gives a weaker range for $q$, namely $q\geq 4$. That said, in \cite{hamming}, an even better range 
\[
q \geq \frac{2(d+1)}{d-1}
\]
is obtained, but it is  conjectured that the sharp range is in fact
\[
q \geq \frac{2d}{d-1}.
\]
Finally, we note that Proposition \ref{converse} and Proposition \ref{hammingformula} together show that the $L^2\to L^q$ extension estimate fails for $q<\frac{2d}{d-1}$.  This was  known to the authors of \cite{hamming}, but it provides another example where relevant parameters can be extracted from using our $L^p$ averages approach.

\section*{Acknowledgements}

We thank Marc Carnovale and Ana de Orellana for helpful discussions on Fourier restriction.

\end{document}